\documentclass{amsart}

\newtheorem{theorem}{Theorem}

\newtheorem{lemma}[theorem]{Lemma}
\newtheorem{corollary}[theorem]{Corollary}

\newtheorem{example}[theorem]{Example}

\theoremstyle{remark}
\newtheorem{remark}[theorem]{Remark}

\numberwithin{equation}{section}

\usepackage{color}

\begin{document}

\title[Smooth surjections and surjective restrictions]
{Smooth surjections and surjective restrictions}

\author[R. M. Aron]{Richard M. Aron }
\address{Richard M. Aron: Department of Mathematical Sciences, Kent State University, Kent, Ohio 44242, USA}
\email{aron@math.kent.edu}

\author[J. A. Jaramillo]{Jes\'us Angel Jaramillo}
\address{Jes\'us A. Jaramillo: Instituto de Matem\'atica Interdiscliplinar (IMI) and Departamento
de An\'alisis Matem\'atico, Universidad Complutense de Madrid, 28040-Madrid, Spain.}
\email{jaramil@mat.ucm.es}

\author[E. Le Donne]{Enrico Le Donne}
\address{Enrico Le Donne: Department of Mathematics and Statistics,
University of Jyv\"{a}skyl\"{a}, Jyv\"{a}skyl\"{a} 40014, Finland.}
\email{ledonne@msri.org}

\thanks{The research of R. M. Aron was supported in part by MINECO (Spain), project MTM2014-57838-C2-2-P. The
research of J. A. Jaramillo was supported in part by MINECO (Spain), project MTM2012-34341. The research of E. Le Donne
was supported by the Academy of Finland, project no. 288501.}


\maketitle

\begin{abstract}
Given a surjective mapping $f : E \to F$  between Banach spaces, we investigate the existence of a subspace $G$ of $E$, with the same density character as $F$, such that the restriction of $f$ to $G$ remains surjective. We obtain a positive answer whenever $f$ is continuous and uniformly open. In the smooth case, we deduce a positive answer when $f$ is a $C^1$-smooth surjection whose set of critical values is countable. Finally we show that, when $f$ takes values in the Euclidean space $\mathbb R^n$, in order to obtain this result it is not sufficient to assume that the set of critical values of $f$ has zero-measure.
\end{abstract}

\begin{section}{Introduction}

In the geometric nonlinear theory of Banach spaces, the study of smooth surjections plays a relevant role. We refer to the book by Benyamini and Lindenstrauss \cite{BL} for extensive information about this subject. One initial
question is about the existence of such smooth surjections. In this direction, it was proved by Bates \cite{Ba2} that
every infinite-dimensional Banach $E$ space admits a $C^1$-smooth mapping $f : E \to F$ onto any separable Banach space $F$. If we look for a  higher degree of differentiability, the situation changes. For example, it is also
proved in \cite{Ba2}  that, if $E$ is superreflexive and ${\rm dens}(E) \geq {\rm dens}(F)$, there exists in fact a $C^{\infty}$-smooth surjection $f : E \to F$. (Here, the {\it density character} of a metric space $X$, denoted by  ${\rm dens}(X)$, is defined as usual as the smallest cardinality of a dense subset of $X$.) Nevertheless, H\'ajek proved in \cite{Ha} that, if $K$ is a countable compact space, there is no $C^2$-smooth surjection from $C(K)$ onto any Banach space with non-trivial type.

\

Among the many important, fundamental results in the theory of linear
operators are so-called {\em selection theorems}. One such is the
Michael selection theorem (see also related work by Bartle and Graves in \cite{BG} or \cite{BL}). One
version of this theorem states that if $T:E \to F$ is a continuous   {\em linear   surjection} between Banach, or
even Fr\'echet, spaces, then there is a continuous (not necessarily linear) mapping $g:F \to E$ such that
$T(g(y)) = y$ for all $y \in F.$ One consequence of this is that if $F$ is, say separable (i.e.
$F$ has a countable dense set), then one can find a separable subspace $G \subset E$ such that $f|_G$
is also surjective. In other words, we can find a (possibly much smaller subspace) of
$E$  having the same density character as  $F$ on which the restriction of $T$ remains an onto mapping.
In this paper, we continue an investigation begun in \cite{AJR} about when this occurs if
the linear surjection $T$ is replaced by a ``good" non-linear one. We will
discuss this question in greater generality, replacing $T$ by an appropriate type of
continuous surjection $f:E \to F$ and relating the existence of such a $G$ to properties of the set of critical values of $f.$
Given a surjection $f:E\to F$ between Banach spaces, where $E$ has larger density character than $F$, we say that $f$ is {\it density-surjective} if there is a closed subspace $G \subset E$ with ${\rm dens}(G) = {\rm dens}(F)$, such that the restriction $f|_G$ remains surjective. Then we wonder which conditions ensure that a smooth surjection is density-surjective.  As we have seen, from the Bartle-Graves selection theorem we have always a positive answer in the linear case. On the other hand, in the general nonlinear case, the answer is negative. Indeed, as follows from \cite{AJR}, for the Hilbert space $E= \ell_2(\Gamma)$, where ${\rm card}(\Gamma) \geq 2^{\aleph_0}$ and for $F = \mathbb R^2$, it is possible to construct a  $C^{\infty}$-smooth surjection $f : E \to F$ such that, for every separable subspace $G$ of $E$, the restriction $f|_G$  is no longer surjective. Furthermore, $f$ is such that ${\rm rank} D f (x)\leq 1$ for every $x\in E$. In fact, the same construction can be carried out for all Banach spaces $E$ which admit a fundamental biorthogonal system with cardinality $\geq 2^{\aleph_0}$ or, more generally, for all Banach space with $C^{\infty}$-cellularity $\geq 2^{\aleph_0}$ (we refer to \cite{AJR} for unexplained terms and for details). The condition about the rank of the mapping $f$ shows the strong failure of the classical Morse-Sard theorem in this infinite-dimensional context (see also \cite{Ba1} for examples in the separable case). Recall that, given a differentiable mapping $f:E\to F$ between Banach spaces, a point $x\in E$ is said to be a {\it regular point} of $f$ if the differential $Df(x): E \to  F$ is surjective. Otherwise we say that $x$ is a {\it critical point} of $f$. A point $y\in F$ is said to be a {\it regular value} of $f$ if its preimage $f^{-1}(y)$ contains no critical points. Otherwise we say that $y$ is a {\it critical value} of $f$. The mapping $f$ is said to be {\it regular} if every value is regular. Thus our mapping $f : \ell_2(\Gamma) \to \mathbb R^2$ is such that every point of $\mathbb R^2$ is a critical value of $f$. This is in contrast with the Morse-Sard theorem, according to which for every $C^k$-smooth mapping $f : \mathbb R^m \to \mathbb R^n$, where $k> \max\{ m-n, 0\}$, the set of critical values of $f$ has zero-measure in $\mathbb R^n$. We refer to \cite{Bo} and \cite{Ku} for classical examples of the failure of the Morse-Sard theorem in infinite-dimensional spaces. As we are going to see, our positive results in this paper will be related to regularity and openness properties of the mapping $f$.

\

The contents of the paper are as follows. In Section 2 we consider the analogous problem in a metric setting. Here we say that a surjection $f:X\to Y$ between metric spaces is {\it density-surjective} if there is a subset $Z \subset X$ with ${\rm dens}(Z) = {\rm dens}(Y)$, such that the restriction $f|_Z$ remains surjective. It is easy to see that, if $X$ and $Y$ are Banach spaces, we recover our previous definition.  Thus we obtain in Theorem \ref{main1} that, if $X$ is complete, every uniformly open surjection $f:X\to Y$ is density-surjective.  In Section 3 we apply this result to the case of smooth surjections between Banach spaces. Using a classical result of Graves (see \cite{Do} and \cite{Gr}) we obtain in Corollary \ref{exceptional} that every  $C^1$-smooth surjection between Banach spaces with a countable number of critical values is density-surjective. For smooth surjections taking values in the Euclidean space $\mathbb R^n$, it is natural to ask if we can obtain an analogous result when the set of critical values has zero-measure in $\mathbb R^n$. A negative answer to this question is given in Theorem \ref{example}. Here we prove that, for the space $\ell_2(\Gamma)$ where ${\rm card}(\Gamma) = 2^{\aleph_0}$, there exists a $C^{\infty}$-smooth surjection $f: \ell_2(\Gamma) \to \mathbb R^2$ whose set of critical values has zero-measure in $\mathbb R^2$ and such that, for every separable subspace $G$ of $\ell_2(\Gamma)$, the restriction $f|_G$  is no longer surjective.

\end{section}

\begin{section}{The metric setting}

In this section we will consider our problem in the context of metric spaces, obtaining a positive result in terms of the following openness condition. We say that a mapping $f:X \to Y$ between metric spaces is {\it uniformly open} if, for every $\varepsilon >0$ there exists some $\delta >0$ so that $B(f(x), \delta) \subset f (B(x, \varepsilon))$, for every $x\in X$. These mappings are also called {\it co-uniformly continuous} in the literature (see e. g. \cite{BL}). A remarkable class of uniformly open mappings are the so-called {\it Lipschitz quotients}, considered for instance in \cite{BJLPS} and \cite{BL}.  Our main result in this section is the following.

\begin{theorem}\label{main1}
Let  $f:X \to Y$ be a continuous, uniformly open surjection between metric spaces, where $X$ is complete.
Then $f$ is density-surjective.
\end{theorem}
\begin{proof}
For each integer $j \geq 0$,  using uniform openness, we can choose some $\delta_j>0$ such that
$$
B(f(x), \delta_j) \subset f (B(x, 2^{-j})), \text{ for every  } x\in X.
$$
 We may also assume that the sequence $(\delta_j)$ is strictly decreasing and $0< \delta_j < 2^{-j}$ for every $j\geq 0$.
Using this, and taking into account that $f$ is surjective, for each $p\in X$, $q\in Y$, and $j\geq 0$, we can choose a
preimage  of $q$ by $f$,  denoted $\sigma_j (p, q) \in X$, with the additional condition that, if $q$ belongs to the ball $B(f(p), \delta_j)$, then $\sigma_j (p, q)$ is in the ball $B(p, 2^{-j})$. In this way we define, for each $j \geq 0$, a map
$\sigma_j : X \times Y \to X$ such that
\begin{enumerate}
\item $f(\sigma_j (p, q)) = q$ for all $p \in X$ and $q \in Y$, and
\item If $d_{Y}(q, f(p)) < \delta_j$, then $d_{X} (\sigma_j(p, q), p) < \frac{1}{2^j}$.
\end{enumerate}
Here we denote by $d_X$ and $d_Y$ the corresponding distances in $X$ and $Y$, respectively.

\

 Now let $\Gamma$ be the first ordinal with the same cardinality as ${\rm dens}(Y)$, and consider a dense set $D$ in $Y$
of this cardinality. We may assume to be in the nontrivial case where $\Gamma$ is not finite, in which case $\Gamma \times \Gamma $
has the same cardinality as $\Gamma$. Using this, and repeating $\Gamma$-times each element of the set $D$, we can define a mapping
$q: \Gamma \to Y$ with the property that the image $q(\Gamma)=D$ is dense in $Y$ and  such that for every $y\in q(\Gamma)$ the preimage $q^{-1}(y)$
has cardinality $\Gamma$, so in particular $q^{-1}(y)$ is cofinal in $\Gamma$. By transfinite induction we are going to select for each ${\gamma< \Gamma}$ some special points in  $f^{-1}(q(\gamma))$. To begin, define $P_0= \{p_0\}$ for a choice of $p_0$ such that $f(p_0)=q(0)$.

Now fix an ordinal $\beta < \Gamma$, with $\beta>0$, and suppose that the subsets $P_{\gamma}$ of $X$  have been defined for each $\gamma < \beta$. Then define the set
$$
P_{\beta} = \left(\bigcup_{\gamma < \beta} P_{\gamma} \right) \bigcup \left\{ \sigma_j (p, q(\beta)) \, : \, p \in \bigcup_{\gamma < \beta} P_{\gamma}; \, j\in \mathbb N\right\}.
$$
In this way we obtain an increasing family $\{P_{\gamma} \}_{ \gamma< \Gamma}$ of subsets of $X$, each with
cardinality not larger than $\Gamma$. Then the cardinality of the set $P= \cup_{ \gamma< \Gamma} P_{\gamma}$ is again not larger than $\Gamma$. Now define $Z=\overline{P}$ to be the closure of $P$ in $X$. It is then clear that ${\rm dens}(Z)\leq {\rm dens}(Y)$. We are going to see that $f(Z)=Y$.

\

 Fix $y\in Y$. Using the density of $D=\{q(\gamma)\}_{\gamma< \Gamma}$, for each $k\in \mathbb N$ we can select
$\gamma_k <\Gamma$ such that the point $q_{\gamma_k} := q(\gamma_k)$ belongs to the ball
$B(y, \frac{1}{2} \delta_k)$. Furthermore, since $q^{-1}(q(\gamma_k))$  is cofinal in $\Gamma$,
we can also assume that $\gamma_k < \gamma_l$ whenever $k<l$.

Now consider $p_0\in X$ selected before; and for every $k\geq 1$ inductively  define
$$
p_{k}= \sigma_{k-1} (p_{k-1}, q_{\gamma_k}).
$$
On the one hand, taking into account  the definition of $P_{\gamma_k}$
and the fact that $\gamma_{k-1} < \gamma_{k}$, one can verify by induction  that $p_k\in P_{\gamma_k}$ for every $k$.
In particular, since the family $(P_\gamma)$ is increasing, we have that the sequence $(p_k)$ is contained in the set $P$. On the other hand,
for every  $k\geq 1$,  from how $ p_k$, $\sigma_{k-1}$, and $q_{\gamma_k}$ have been defined, we have that
\begin{eqnarray*}
d_Y (q_{\gamma_{k+1}}, f(p_k)) &=& d_Y (q_{\gamma_{k+1}}, f(\sigma_{k-1} (p_{k-1}, q_{\gamma_k})))\\
&=& d_Y (q_{\gamma_{k+1}}, q_{\gamma_k}) \\
&\leq& d_Y (q_{\gamma_{k+1}}, y)  +  d_Y (y, q_{\gamma_k})\\
&\leq&\frac{\delta_{k+1}}{2}+ \frac{\delta_k}{2}< \delta_k.
\end{eqnarray*}
 Because of the second property of the definition of
  $\sigma_{k-1}$, we deduce that
 $$
d_X (p_k,p_{k+1}  ) =
d_X (p_k, \sigma_k (p_k, q_{\gamma_{k+1}})) < \frac{1}{2^k}.
$$
We conclude that   $(p_k)$ is a Cauchy sequence in $P$, and by completeness of $X$ it  converges to some  $p\in \overline{P}= Z$. By continuity, the sequence $(f(p_k))$ converges to $f(p)$. But we have that $f(p_k) = f( \sigma_{k-1}(p_{k-1}, q_{\gamma_k})) = q_{\gamma_k}$ which converges  to $y$. In this way we obtain that $y=f(p)$.

\

Then we have obtained that $f|_Z: Z \to Y$ is a continuous surjection and ${\rm dens}(Z) \leq {\rm dens}(Y)$, so we have in fact that  ${\rm dens}(Z) = {\rm dens}(Y)$.

\end{proof}

\begin{remark} The above result is trivial when ${\rm card}(Y)={\rm dens}(Y)$. In fact, in this case every surjection $f: X \to Y$ is density-surjective. Indeed, we only need to choose a preimage in $X$ for each point of $Y$. An example of this situation is the Banach space $Y= \ell_{\infty}$, where we have that ${\rm card}(\ell_{\infty})={\rm dens}(\ell_{\infty})= 2^{\aleph_0}$. On the other hand, for certain nonseparable spaces, the proof of Theorem \ref{main1} can be simplified. More precisely, suppose that $X$ is a complete metric space and $Y$ is a metric space whose density character $\alpha ={\rm dens}(Y)$ satisfies $\alpha^{\aleph_0} = \alpha$. Consider a dense subset $D$ of $Y$ with cardinality $\alpha$. Then the set ${\mathcal S}(D)$ of all sequences in $D$ which are convergent in $Y$ has cardinality $\alpha^{\aleph_0} = \alpha$. Now fix a point $p_0\in X$. For each sequence $s=(q_k)_{k\in \mathbb N} \in {\mathcal S}(D)$ we can define the associated sequence $p^s=(p^s_k)_{k\in \mathbb N}$ in $X$ given by $ p^s_{k}= \sigma_{k-1} (p^s_{k-1}, q_k)$, for every $k\geq 1$. Proceeding as in the proof of Theorem \ref{main1} we obtain that the sequence $p^s$ is convergent in $X$. Furthermore, the set $Z$ of limits of all sequences $p^s$, where $s \in {\mathcal S}(D)$, satisfies that ${\rm dens}(Z)=\alpha ={\rm dens}(Y)$ and $f|_Z: Z \to Y$ is surjective. Nevertheless,  the above remarks cannot be applied to every nonseparable space. For instance, it is well-known that, as a consequence of K{\"o}nig's Theorem, $\aleph_{\omega}^{\aleph_0} > \aleph_{\omega}$ (see e.g. \cite{Hau}, page 40 or \cite{Kun}, page 34). Now choose a set $\Gamma$ with ${\rm card}(\Gamma)= \aleph_{\omega}$ and consider the Hilbert space $Y=\ell_2(\Gamma)$. It is easily seen that in this case ${\rm dens}(Y)= \aleph_{\omega}$ and ${\rm card}(Y)=\aleph_{\omega}^{\aleph_0}$.
\end{remark}

\end{section}

\begin{section}{Regular mappings between Banach spaces}

Recall that a differentiable mapping $f: E \to F$ between Banach spaces is said to be {\it regular at point} $x_0\in E$ if the differential $Df(x_0): E \to F$ is onto. If this holds for every point of $E$, we say simply that $f$ is {\it regular}. The connection between regularity and openness is given by the following classical result due to Graves (see \cite{Gr} or \cite{Do}).

\begin{theorem} {\rm (Graves)}\label{Graves}
Let $f: E \to F$ be a $C^1$-smooth mapping between Banach spaces, and suppose that $f$ is regular at a point $x_0\in E$. Then there exist a neighborhood $U$ of $x_0$ and a constant $c>0$ such that for every $x\in U$ and every $\tau >0$ with $B(x, \tau)\subset U$, we have that
$$
B(f(x), c \tau)\subset f( B(x, \tau)).
$$
\end{theorem}

Using Theorem \ref{Graves} we obtain our next corollary.

\begin{corollary}\label{exceptional}
Let $f: E \to F$ be a $C^1$-smooth surjection between Banach spaces.  If the set of critical values of $f$ has cardinality $\leq {\rm dens}(F)$, then $f$ is density-surjective. In particular this applies when the set of critical values of $f$ is countable.
\end{corollary}
\begin{proof}
Let $E_0$ denote the set of regular points of $f$. Then $C=f(E \setminus E_0)$ is the set of critical values of $f$. By  hypothesis  ${\rm card}(C)\leq {\rm dens}(F)$, so there is a  subset $H\subset E \setminus E_0$ with ${\rm card}(H) \leq {\rm dens}(F)$ and such that $f(H) = C$.

Now for each $m , k \in \mathbb N$, consider the set
$$
E_{m, k} = \left\{ x\in E \, : \, B\left(f(x), \frac{\tau}{m} \right) \subset f \left( B(x, \tau) \right), \, \text{  for every  } \tau \in \left(0, \frac{1}{k}\right) \right\}.
$$
Define $X_{m, k}= \overline{E_{m, k}}$ to be the closure of $E_{m, k}$ in $E$ and $Y_{m, k}= f( X_{m, k})$.

Note that from Graves' Theorem we have that
$$
E_0 \subset \bigcup_{m, k} E_{m, k}.
$$
and therefore
$$
F= C\cup \bigcup_{m, k} Y_{m, k}.
$$
For each $m , k \in \mathbb N$ we have then a continuous surjection $f_{m, k}= f|_{X_{m, k}} \to Y_{m, k}$, and we are going to show that $f_{m, k}$ is uniformly open. Indeed, let $x\in X_{m, k}$ be given, consider $0<\tau < \frac{1}{k}$ and let $y\in B( f(x), \frac{\tau}{m})$. We know that there exists a sequence $(x_j)$ in $E_{m, k}$ converging to $x$. Thus $(f(x_j))$ converges to $f(x)$ and for $j$ large enough we have that $\Vert y-f(x_j)\Vert < \frac{\tau}{m}$. Then
$$
y\in B( f(x_j), \frac{\tau}{m}) \subset f (B(x_j, \tau)),
$$
so there exists some $u_j \in B(x_j, \tau)$ such that $y=f(u_j)$. Taking $j$ large enough we also have that $\Vert x_j -x\Vert < \tau$, and then $\Vert u_j -x \Vert \leq \Vert u_j -x_j \Vert + \Vert x_j -x\Vert< 2 \tau$. In this way $y= f(u_j) \in f (B(x, 2 \tau))$. This shows that
$$
B(f(x), \frac{\tau}{m} ) \subset f(B(x, 2 \tau)).
$$
Now, for each $\varepsilon >0$ we choose $\tau$ with $0< 2\tau < \min \{\varepsilon, \frac {1}{k}\}$  and we take $\delta = \frac{\tau}{m}$. Then we obtain that
$$
B(f(x), \delta) \subset f(B(x, 2 \tau)) \subset f(B(x, \varepsilon)),
$$
for every $x\in X_{m, k}$.

As a consequence of Theorem \ref{main1} we deduce that, for each $m , k$   there exists a subset $Z_{m, k}$ of $X_{m, k}$ such that ${\rm dens} (Z_{m, k}) = {\rm dens} (Y_{m, k})$ and $f(Z_{m, k}) = Y_{m, k}$. We know that every subset of a metric space has density not larger than the whole space. Thus ${\rm dens} (Z_{m, k}) \leq {\rm dens} (F)$ for every $m, k$. If we now choose
$$
G=\overline{[H \cup \bigcup_{m, k} Z_{m, k}]}
$$
to be the closed linear span of $H \cup \bigcup_{m, k} Z_{m, k}$, we have that ${\rm dens} (G) \leq {\rm dens} (F)$ and $f(G)=F$.

\end{proof}

\begin{example} {\rm We point out that there exist smooth surjections defined on non-separable Banach spaces, whose set of critical values is countably infinite. For a simple example, consider the space $E=\ell_{\infty}$ and let $f: \ell_{\infty} \to \mathbb R$ be the mapping defined for each $x=(x_n)\in \ell_{\infty}$ by
$$
f(x)= x_1 + \cos (x_1) + \sum_{n=2}^{\infty} \frac{x_n^3}{2^n}.
$$
It is easy to see that $f$ is a $C^\infty$-smooth surjection and, for every $x=(x_n)$ and  $u=(u_n)$ in $\ell_{\infty}$, we have that
$$
Df(x)(u) = (1-\sin (x_1))\, u_1 + \sum_{n=2}^{\infty} \frac{3 x_n^2}{2^n} \,u_n.
$$
Then $Df(x)=0$ if, and only if, the point $x=(x_n)$ satisfies  $\sin (x_1)=1$ and $x_n=0$ for every $n\geq 2$. We obtain that the set of critical values of $f$ is
$$
CV(f)= \left\{ \frac{\pi}{2} + 2k \pi \, : \, k \in \mathbb Z \right\}.
$$
In order to construct examples of vector-valued surjections, we proceed as follows. First consider a separable quotient $F$ of $\ell_{\infty}$. This means that $F$ is a separable Banach space and there exists a continuous linear surjection  $T: \ell_{\infty} \to F$. Of course $F$ can be chosen to be $F=\mathbb R^n$, but also  we can choose $F$ to be $F= \ell_2$ (see for example the Remarks
on page 111 of \cite{LT}). By composing with a linear isomorphism of $\ell_{\infty}$ if necessary, we may also assume that $T(e)=0$, where  $e= (1, 1, 1, \cdots) \in \ell_{\infty}$. Now define $g: \ell_{\infty} \to \mathbb R \times F$  by
$$
g(x) = (f(x), T(x)).
$$
Let us see that $g$ is surjective. Given $(\lambda, w)\in  \mathbb R \times F$, since $T$ is surjective there exists some $v=(v_n) \in \ell_{\infty}$ such that $T(v)=w$. Note that, for every $t \in \mathbb R$, we have that $T(v + t e)=w$ and
$$
f(v + t e)= v_1+ t + \cos(v_1+t) + \sum_{n=2}^{\infty} \frac{(v_n+t)^3}{2^n}.
$$
It is easily seen that the function $\phi: \mathbb R \to \mathbb R$ given by $\phi (t)=f(v + t e)$ is surjective, and therefore there exists some $t \in \mathbb R$ such that $g(v + t e)= (\lambda, w)$.
On the other hand, it is clear that $g$ is  $C^\infty$-smooth and, for each $x,u \in \ell_{\infty}:$
$$
Dg(x)(u) = (Df(x)(u), T(u)).
$$
Now we are going to check that, if $Df(x)\neq 0$, then $Dg(x):\ell_{\infty} \to \mathbb R \times F$ is surjective. Indeed, given $(\lambda, w)\in  \mathbb R \times F$, choose $v\in \ell_{\infty}$ such that $T(v)=w$. Since $Df(x)\neq 0$ we have that $Df(x)(e)\neq 0$, and therefore the function $\psi: \mathbb R \to \mathbb R$ given by $\psi (t)=Df(x)(v + t e)$ is surjective. Then there exists some $t \in \mathbb R$ such that $Dg(x)(v + t e)= (\lambda, w)$. Thus we obtain that $x$ is a critical point of $g$ if, and only if, $Df(x)=0$.  The set of critical values of $g$ is then
$$
CV(g)= \left\{ \left(\frac{\pi}{2} + 2k \pi, T(a_k) \right) \, : \, k \in \mathbb Z \right\},
$$
where $a_k= (\frac{\pi}{2} + 2k \pi, 0, 0, \cdots)\in \ell_{\infty}$, for each $k\in \mathbb Z$. Finally note that, as a consequence of Corollary \ref{exceptional}, the mapping $g$ is density-surjective.}

\end{example}
\qed

\

Of course, for a regular surjection the set of critical values is empty, and we obtain the following:

\begin{corollary}\label{regular}
Every  regular, $C^1$-smooth surjection between Banach spaces is density-surjective.
\end{corollary}

A relevant class of regular surjections are the so-called strong submersions, considered by Rabier in \cite{Ra} (see also \cite{GJ}) in the more general context of Banach-Finsler manifolds.

\

A $C^1$-smooth mapping  $f: E \to F$ between Banach spaces is said to be a {\it strong submersion} if there is no sequence $(x_n)$ in $E$ such that $(f(x_n))$ is convergent in $F$ and $\lim_n \nu(Df(x_n))=0$. Here $\nu(T)$ is defined for every continuous linear operator $T:E \to F$ as
$$
\nu(T)=\inf \{\Vert T^* (y^*) \Vert \, : \, y^* \in F^*, \Vert y^* \Vert = 1 \},
$$
where $T^* :F^* \to E^*$ is the adjoint operator of $T$. Note that this implies in particular that $Df(x)$ is surjective, for every $x\in E$.

\

On the other hand, also following \cite{Ra} (see Definition 3.1), a $C^1$-smooth mapping  $f: E \to F$ between Banach spaces is said to have {\it uniformly split kernels} if there is a constant $C\geq 1$ such that,  for every $x\in E$, there exists a continuous linear projection $P_x : E \to ker Df(x)$ with $\Vert P_x\Vert \leq C$. Note that this condition is automatically satisfied if either $E$ is a Hilbert space or $F$ is finite-dimensional.

\

For strong submersions with uniformly split kernels a global implicit function theorem is obtained in \cite{Ra}, which in particular provides the existence of a continuous section in this case. More precisely, we have the following result.
\begin{theorem} {\rm (Theorem 5.2 in \cite{Ra})}
Let $f: E \to F$ be a $C^1$-smooth mapping between Banach spaces, with locally Lipschitz derivative. Suppose that $f$ is a strong submersion with uniformly split kernels. Then:
\begin{enumerate}
\item These exist a closed subset $W$ of $E$ (in fact, a closed $C^1$-submanifold), and a homeomorphism $\theta : W\times F \to E$ such that
$f(\theta (w, y))=y$ for every $(w, y) \in W\times F$.
\item There exists a continuous section $\varphi : F\to E$ of $f$, and therefore $f$ is density-surjective.
\end{enumerate}
\end{theorem}
\begin{proof}
Part $(1)$ is contained in Theorem 5.2 of \cite{Ra}. In order to obtain $(2)$, fix a point $w\in W$ and define $\varphi (y)= \theta (w, y)$. Then  $\varphi: F \to E$ is continuous and  $f(\varphi (y))= y$ for every $y\in F$. Now  if we consider $G=\overline{[\varphi (F)]}$ the closed linear subspace of $E$ spanned by $\varphi (F)$, we have that $G$ has the same density character as $F$ and $f|_G$ is surjective.
\end{proof}

We have obtained in Corollary \ref{exceptional} a positive  density-surjection result when the mapping $f$ is a smooth surjection with a countable number of critical values. So we may further ask about the size of the set of critical values, especially in the case of mappings taking values in Euclidean spaces. In particular, we may wonder if we could have a positive answer to our problem when the conditions of the Morse-Sard theorem are fulfilled. More precisely, suppose that $f: E \to \mathbb R^n$ is a $C^{\infty}$-smooth surjection from a non-separable Banach space $E$, such that the set of critical values of $f$ has zero-measure in $\mathbb R^n$. Is there a separable subspace $G$ of $E$  such that the restriction $f|_G$ remains surjective? The following example provides a negative answer to this question (cf. Corollary 7 of \cite{AJR}).

\begin{theorem}\label{example}
Let  $\Gamma$ be a set with ${\rm card}(\Gamma) = 2^{\aleph_0}$. There exists a $C^{\infty}$-smooth surjection $f: \ell_2(\Gamma) \to \mathbb R^2$ such that:
\begin{enumerate}
\item For every separable subspace $G$ of $\ell_2(\Gamma)$, the restriction $f|_G$  is no longer surjective.
\item The set of critical values of $f$ has zero-measure in $\mathbb R^2$.
\end{enumerate}
\end{theorem}
\begin{proof}
 For our construction, we will need some auxiliary mappings on the plane. Fix $0<a<b<c$ and $m\in \mathbb N$, and
define a mapping
$$
 g_{(a, b, c, m)} : \mathbb R^2 \to \mathbb R^2
$$
as follows.  {First choose a $C^\infty$-smooth function $\varphi_m: \mathbb R^2 \to \mathbb R$ such that $0\leq \varphi_m \leq 1$,  $\varphi_m = 1$ identically on the unit square $[0, 1]\times [0,1]$ and the support of $\varphi_m$ is contained in the square  $[-r_m, r_m]\times [-r_m, r_m]$, where $r_m=\sqrt{\frac{m+1}{m}}$.}

 {Next, consider a $C^\infty$-smooth function $\theta : \mathbb R \to \mathbb R$ with $0\leq \theta \leq 1$, such that
$\theta (t) =0$ for $t\leq a$ and $\theta (t)=1$ for $t \geq b$. Then define
$$
g_{(a, b, c, m)}(x, y) = (c x^2 \, \varphi_m(x, y), \, m y^2 \, \theta (c x^2 \varphi_m (x, y)) \, \varphi_m (x, y)).
$$
We then have that $g_{(a, b, c, m)}$ is  $C^\infty$-smooth with compact support contained in the Euclidean ball  $B((0, 0), 2)$, and its image $E(a, b, c, m) = g_{(a, b, c, m)}(\mathbb R^2)$ satisfies that
$$
[b, c]\times [0, m] \subset E(a, b, c, m) \subset ([0, a] \times \{0\}) \cup ([a, \frac{m+1}{m}c] \times [0, m+1]).
$$
Indeed, for the first containment, if $(x, y)$ belongs to $[\sqrt{\frac{b}{c}}, 1] \times [0, 1]$ then $\varphi_m (x, y)=1$ and $\theta (c x^2 \varphi_m (x, y))=1$, so that $g_{(a, b, c, m)}(x, y)= (c x^2, m y^2)$. Concerning the second containment, note  that $g_{(a, b, c, m)}(x, y)$ is always contained in $[0, \frac{m+1}{m}c]\times [0, m+1]$. Now, if the first coordinate $c x^2 \, \varphi_m(x, y)$ belongs to the interval $[0, a]$ then $\theta$ vanishes on it and the second coordinate is zero.}

  {Furthermore, note that an analogous mapping  $g_{(a, b, c, m)}$ can be constructed for $m\in \mathbb Z\setminus \{0\}$ and also for $0>a>b>c$.}

\

Now let $C$ be the usual Cantor set in $\mathbb R$ and consider the set $C\times \mathbb R$. It is not difficult to see that the set
$$
A=(\mathbb R^2 \setminus (C\times \mathbb R)) \cup (\mathbb R \times \{0\})
$$
can be written as a countable union of sets of the above form:
$$
A= \bigcup_{j\in \mathbb N} E(a_j, b_j, c_j, m_j)
$$
for some $a_j, b_j, c_j \in \mathbb Q$ and some $m_j \in \mathbb Z$. Now for each $j \in \mathbb N$, using translations and homotheties we can easily modify the corresponding mapping  $g_{(a_j, b_j, c_j, m_j)}$ in order to obtain a mapping $\tilde{g}_j$ with support contained in the Euclidean ball $B((j, 0), \frac{1}{4})$ and such that  $\tilde{g}_j(\mathbb R^2)= E(a_j, b_j, c_j, m_j)$. We now glue together these mappings and define
$$
g : \mathbb R^2 \to \mathbb R^2
$$
by setting $g(x, y) = \tilde{g}_j (x, y)$ if $(x, y)\in B((j, 0), \frac{1}{4})$ and $g(x, y)=0$ otherwise. We obtain that  $g$ is
$C^\infty$-smooth and $g( \mathbb R^2) = A$. Furthermore, by the Morse-Sard theorem, we know that the set of critical values of $g$ has zero-measure in $\mathbb R^2$.

\

Consider now the Hilbert space $H=\mathbb R^2 \times \ell_2(\Gamma)$, which is in fact isomorphic to $\ell_2(\Gamma)$. Denote by $(e_\gamma)_{\gamma \in \Gamma}$ the usual orthonormal basis of  $\ell_2(\Gamma)$. The next part of the construction will be similar to  the proof of Theorem 4 in \cite{AJR}. For each $\gamma \in \Gamma$, we consider the open set $W_{\gamma} \subset H$ given by the product of open balls $W_{\gamma}= B((0, 0), \frac{1}{4}) \times B(e_{\gamma}, \frac{1}{4})$ and, by using the smoothness properties of $\ell_2(\Gamma)$ (see e. g. \cite{DGZ}) we can find a function $\phi_{\gamma}\in C^{\infty}(H)$ with support contained in $W_{\gamma}$, such that $0 \leq \phi_{\gamma} \leq 1$ and
$\phi_{\gamma} (e_{\gamma})=1$. Let $\eta:\mathbb R \to \mathbb R$ be a $C^{\infty}$-smooth function with
$0\leq \eta\leq 1$, $\eta (t)=0$ for every $t\leq\frac{1}{6}$
and $\eta(t)=1$ for every $t\geq\frac{1}{5}$. Now if we define
$g_{\gamma}=\eta \circ \phi_{\gamma}$, we obtain that
$g_{\gamma}\in C^{\infty}(H)$, the support of $g_{\gamma}$ is
contained in $W_{\gamma}$, and $g_\gamma = 1$ on the nonempty open subset
$V_{\gamma} = (\phi_{\gamma})^{-1}(\frac{1}{5}, \infty)$ of $W_{\gamma}$. Next,  for each
$k \in \mathbb Z$ we choose a $C^{\infty}$-smooth function $\theta_k:\mathbb R \to \mathbb R$ such that
$\theta_{k}(t)=0$ for every $t\leq\frac{1}{4}$ and
$\theta_{k}([\frac{1}{3}, \frac{1}{2}])=[k-\frac{1}{2},
k+\frac{1}{2}]$, and  we define $h_{\gamma, k}=\theta_{k}\circ \phi_{\gamma}$. Then $h_{\gamma, k} \in C^{\infty}(H)$
has support contained in $V_{\gamma}$ and the image $h_{\gamma, k}
(V_{\gamma})$ contains the interval $[k-\frac{1}{2},
k+\frac{1}{2}]$.

Consider a partition $\Gamma = \cup_{k\in \mathbb Z} \Gamma_k$, where for each
$k\in \mathbb Z$ the set $\Gamma_k$ has cardinality $2^{\aleph_0}$,
and for  each $k$ choose a bijection with the Cantor set $\sigma_k: \Gamma_k \to C$.

For each $\gamma \in \Gamma,$ choose the unique $k$ such that $\gamma \in \Gamma_k,$ and define
$f_\gamma:H \to \mathbb R^2$ by setting
$$
f_\gamma(x)=( g_\gamma(x) \cdot \sigma_k (\gamma), \,  h_{\gamma, k}(x)),
$$
if $x$ belongs  $W_{\gamma}$; and $f_\gamma(x)=(0,0)$ otherwise.
Then each $f_\gamma$ is $C^{\infty}$-smooth on $H$,  with support contained into
$W_{\gamma}$. Note that since $\gamma \in \Gamma_k$ we have that
$f_{\gamma}|_{V_{\gamma}}(x)= (\sigma_k(\gamma), h_{\gamma, k}(x))$, and
therefore the image $f_{\gamma}(V_{\gamma})$ contains the set
$\{\sigma_k(\gamma)\}\times [k-\frac{1}{2}, k+\frac{1}{2}]$.

\

Next, for each $j\in \mathbb N$ consider the set $U_j= \pi^{-1}(B((j, 0), \frac{1}{4}))$, where $\pi : H=\mathbb R^2 \times \ell_2(\Gamma) \to \mathbb R^2$ is the natural projection. Since the family  $\{W_{\gamma} \}_{\gamma \in \Gamma} \cup \{U_j\}_{j \in \mathbb N}$ is pairwise disjoint and locally finite in $H$, we can define a $C^\infty$-smooth mapping $f:H
\to \mathbb R^2 \times \mathbb R$ by setting
$$
f(x)= f_{\gamma}(x), \text{  if  } x\in W_{\gamma} \text{  for some  } \gamma \in \Gamma;
$$
$$
f(x)= g(\pi (x)), \text{  if  } x\in U_j \text{  for some  } j \in \mathbb N;
$$
$$
f(x)=(0,0), \text{   otherwise.}
$$
By the construction of the functions $f_{\gamma}$, we have that $f(\cup_{\gamma \in \Gamma} W_{\gamma})= (C\times \mathbb R) \cup (\mathbb R \times \{0\})$.
Now, from the construction of  $g$, it is clear that also the complement $\mathbb R^2 \setminus (C\times \mathbb R)$ is contained in the image of $f$. Therefore, $f$ is a surjection.

On the other hand, suppose that $X$ is a subset of $H$ such that
$f|_X:X\to C \times \mathbb R$ is onto. Then for each $u \in C$
there exists some $x_u\in H$ such that $f(x_u)=(u,1)$. By the
construction, since the second coordinate of $f(x_u)$ is nonzero,
there exists some $\gamma \in \Gamma$ such that $x_u\in V_{\gamma}$.
Then $\gamma \in \Gamma_k$ for some $k$ and by the definition of
$f_\gamma$ we have that $v=\sigma_k(\gamma)$. Therefore
$$
C =\bigcup_{k \in \mathbb Z} \sigma_k(\{\gamma \in \Gamma_k \,: \, X \cap V_{\gamma}  \neq \emptyset \}).
$$
From a
cardinality argument we obtain that there exists some $k\in \mathbb
Z$ such that
$$
\text{ card } \{\gamma \in \Gamma_k \, : \, X \cap V_{\gamma} \neq \emptyset\}
$$
is uncountable. Thus $X$ contains an  uncountable family of non-empty, disjoint, open sets, so  $X$ cannot be separable.

\

Finally, let $N$ denote the set of critical values of $g$. We are going to check that the set of critical values of $f$  is contained in $N \cup (C\times \mathbb R) \cup (\mathbb R \times \{0\})$ and therefore has zero-measure in $\mathbb R^2$. Indeed,  consider a point
$(u, v)\in \mathbb R^2 \setminus (N \cup (C\times \mathbb R) \cup (\mathbb R \times \{0\}))$. For every $x \in H$ with $f(x)= (u, v)$ there exists some $j\in \mathbb N$ such that $x\in U_j$ and $f= g\circ \pi$ on a neighborhood of $x$.  Furthermore, $g$ is regular at $\pi (x)$. Thus $f$ is regular at $x$.

\end{proof}

\

To conclude, we remark that an adaptation of the above argument yields a $C^{\infty}$-smooth surjection $f: \ell_2(\Gamma) \to \mathbb R^2$ such that to every point $y \in \mathbb R^2$
there corresponds $x \in \ell_2(\Gamma)$ such that $f(x) = y$ and also $Df(x) = 0.$

\end{section}

\bigskip

\centerline {\sc Acknowledgements:}
\vspace{.2cm}

Different parts of this paper were developed during a visit of Enrico Le Donne to ICMAT (Madrid) and a visit of Jes\'us A. Jaramillo to Kent State University. We would like to  thank these institutions  for their kind hospitality and support.  It is also a pleasure to thank Patrick Rabier for several valuable comments concerning this paper.

\bigskip

  {\bf Post Scriptum:} The recent work \cite{JLR} contains a proof of a more general
statement than Theorem \ref{main1}.


\end{document}